\documentclass[letterpaper,11pt,oneside,reqno]{amsart}
\usepackage{amsmath,amsthm,amsfonts,amssymb,bbm,color}
\usepackage{graphicx,psfrag,subfigure,url}
\usepackage{cite}
\usepackage{mathrsfs}
\usepackage[DIV13]{typearea}
\usepackage{comment}
\usepackage{caption}

\usepackage{color}
\usepackage[colorlinks=true,
            linkcolor=red,
            urlcolor=blue,
            citecolor=blue]{hyperref}

\numberwithin{equation}{section}

\newcommand{\I}{{\iota}}

\newcommand{\EE}{\ensuremath{\mathbb{E}}}
\newcommand{\PP}{\ensuremath{\mathbb{P}}}
\newcommand{\R}{\ensuremath{\mathbb{R}}}

\newcommand{\Z}{\ensuremath{\mathbb{Z}}}

\newcommand{\f}{\mathbf{f}}

\newcommand{\barspace}{\ensuremath{\mid}}
\newcommand*\pFqskip{8mu}
\catcode`,\active
\newcommand*\pFq{\begingroup
        \catcode`\,\active
        \def ,{\mskip\pFqskip\relax}%
        \dopFq
}
\catcode`\,12
\def\dopFq#1#2#3#4#5#6{%
        {}_{#1}\phi_{#2}\biggl(\genfrac..{0pt}{}{#3}{#4};#5;#6\biggr)%
        \endgroup
}

\newtheorem{theorem}{Theorem}[section]

\newtheorem{lemma}[theorem]{Lemma}
\newtheorem{corollary}[theorem]{Corollary}

\newtheorem{remark}[theorem]{Remark}

\newtheorem{definition}[theorem]{Definition}
\newenvironment{remarks}{\begin{remark}\normalfont}{\end{remark}}

\title[Dynamic ASEP, duality and continuous $q^{-1}$-Hermite polynomials]{Dynamic ASEP, duality and continuous $q^{-1}$-Hermite polynomials}

\author[A. Borodin]{Alexei Borodin}
\address{A. Borodin,
Massachusetts Institute of Technology,
Department of Mathematics,
77 Massachusetts Avenue, Cambridge, MA 02139-4307, USA}
\email{borodin@math.mit.edu}

\author[I. Corwin]{Ivan Corwin}
\address{I. Corwin, Columbia University,
Department of Mathematics,
2990 Broadway,
New York, NY 10027, USA}
\email{ivan.corwin@gmail.com}

\begin{document}

\begin{abstract}
We demonstrate a Markov duality between the dynamic ASEP and the standard ASEP. We then apply this to step initial data, as well as a half-stationary initial data (which we introduce). While investigating the duality for half-stationary initial data, we uncover and utilize connections to continuous $q^{-1}$-Hermite polynomials. Finally, we introduce a family of stationary initial data which are related to the indeterminate moment problem associated with these $q^{-1}$-Hermite polynomials.
\end{abstract}

\sloppy \maketitle
%\setcounter{tocdepth}{1}
%\tableofcontents

\section{Introduction}

The first main result of this paper demonstrates a Markov duality between the standard asymmetric simple exclusion process \cite{Liggett} (ASEP) and the dynamic ASEP -- a generalization of ASEP with an extra dynamic parameter $\alpha$ which was recently introduced in \cite{IRF} and studied via elliptic generalizations of constructions developed in \cite{BorodinR} and \cite{BorodinPetrovInhomo} (see also \cite{BP-lectures}) to study the (trigonometric) higher-spin six vertex models. We will not make much direct use of this general construction or methods here, and rather seek to develop a more probabilistic understanding of the dynamic ASEP through proving (in Theorem \ref{thm.duality}) a Markov duality.

Markov dualities for the (trigonometric) higher-spin six vertex models have been shown in \cite{CorwinPetrov}, as well as in special cases in earlier works such as \cite{Schutz, BCS, BCdiscrete, CorwinHahn}. This was one motivation for our present investigation. The other came from the fact that \cite[Theorem 10.1 and Corollary 10.6]{IRF} showed that certain expectations (involving our duality function) for step initial data dynamic ASEP are independent of the dynamic parameter $\alpha$. Duality extends this situation to general initial data by showing that the evolution of the expectation of the duality function evaluated along the trajectory of dynamic ASEP is itself $\alpha$-independent. The formulas already derived for step initial data dynamic SSEP (the $q\to 1$ limit of dynamic ASEP) have been useful for asymptotics \cite{IRF,Amol}, and we expect that the duality will prove useful in extending this type of asymptotics to more general initial data.

For  stochastic (trigonometric) higher-spin vertex models and their degenerations (as well as other particle systems which do not fit into that hierarchy but also enjoy similar dualities -- e.g. \cite{Giardina, Kuan, SchutzBelinsky}) duality has proved to be a useful tool. For instance, the $n=1$ duality for those models implies that $q$ raised to the height function satisfies a microscopic stochastic heat equation. This is a microscopic version of the Cole-Hopf transform (also known as the G\"{a}rtner transform in the context of ASEP) which transforms the Kardar-Parisi-Zhang (KPZ) equation into the continuum stochastic heat equation with multiplicative space-time white noise \cite{ICReview} and serves as the starting point for proving convergence under certain scalings of these discrete models to the KPZ equation \cite{BG, ACQ, DemboTsai, CorwinTsai, CorwinShenTsai, CorwinShen, Promit}.

The $n=1$ version of our duality implies that a certain quadratic transformation of $q$ raised to the height function satisfies a microscopic stochastic heat equation. This observation is a possible starting point to try to study certain stochastic PDE limits of the dynamic ASEP. Indeed, our study in Section \ref{sec.stat} of stationary initial data for dynamic ASEP is also useful in this pursuit as it helps to identify the non-trivial scalings one can take, as well as the measures which must remain stationary for the limiting SPDE. Our investigation of the stationary measures suggests that the limit of the associated microscopic stochastic heat equation for dynamic ASEP must preserve a quadratic transform of exponentiated spatial Ornstein-Uhlenbeck process. This, in turn, suggests that the limiting noise for the continuum stochastic heat equation may not be multiplicative. We do not pursue this direction more herein.

%In general, dualities reduce the computation of certain expectations of infinite particle systems into computations of expectations of particle systems with finitely many particles. The complexity of the infinite system gets embedding into the initial data of the dual finite particle system.

The duality we demonstrate here has a limit which becomes a variant of the ASEP self-duality demonstrated in \cite{BCS}. That self-duality generalizes to the top of the hierarchy of stochastic higher-spin six vertex models \cite{CorwinPetrov}. The dynamic analog of those models have recently been introduced and studied in \cite{Amol}. We anticipate that our duality may also similarly lift to those dynamic stochastic higher spin vertex models. It would also be interesting to try to adapt the methods of \cite{Giardina,Kuan} to find new (possibly higher rank, or multi-species) systems which enjoy similar dualities.

The second main result of this paper concerns connecting dynamic ASEP with continuous $q^{-1}$-Hermite polynomials -- special cases of $q$-orthogonal polynomials which fit into the $q$-deformation of the Askey-scheme -- see \cite{GasperRahman, KoekoekSwarttouw}. This connection (whose deeper meaning is yet to be understood) manifests itself in two ways. The first comes in our study of {\it half-stationary} initial data for dynamic ASEP. This initial data arises as the spatial trajectory of a dynamic nearest neighbor random walk (in the sense that the jump probabilities depend on the height) and the one-step transition matrix for this random walk ends up begin diagonal in the basis of continuous $q^{-1}$-Hermite polynomials. (Random walks with similar properties, but with respect to classical orthogonal polynomials, have been studied, see e.g. \cite{Persi}.) In order to apply duality to this initial data, we must compute the expectation of our duality functional at time zero with respect to this (random) initial data. That computation eventually boils down to the diagonalization result just mentioned, and a simple summation identity \eqref{eq.sumid} for these polynomials.

The second manifestation of the connection to these polynomials is in our study of {\it stationary} initial data for dynamic ASEP which we show is related to the known classification of solutions to the indeterminate moment problem for the weight associated with these polynomials \cite{IM}. As the study of dynamic stochastic higher-spin vertex models advances, it will be interesting to see how high in the Askey-scheme these connections go.
%Though there is no indication of a relationship, we note that the Askey-Wilson polynomials have also arisen in the study of (standard) ASEP with open boundaries on a finite line segment, as in \cite{UchiyamaSasamotoWadati,CorteelWilliamsDuke}.

\subsubsection*{Outline}
Section \ref{sec.main} contains our main results, namely Theorem \ref{thm.duality} (Markov duality), Theorem \ref{thm.initialdata} (step and half-stationary initial data evaluation formulas), and Theorem \ref{thm.stationary} (stationary initial data). Section \ref{sec.thmproof} contains the proof of Theorem \ref{thm.duality}, and Section \ref{sec.half} contains the proof of Theorem \ref{thm.initialdata} along with the proofs of Lemma \ref{lem.eign} and \ref{sumid}. The continuous $q^{-1}$-Hermite polynomials play a key role in Theorem \ref{thm.stationary} as well as the proofs of Lemma \ref{lem.eign} and \ref{sumid}.

\subsubsection*{Acknowledgements}
I. Corwin was partially supported by the Clay Mathematics Institute through a Clay Research Fellowship, and the Packard Foundation through a Packard Fellowship for Science and Engineering.
A. Borodin was partially supported by the NSF grants DMS-1056390 and DMS-1607901, and by a Radcliffe Institute for Advanced Study Fellowship, and a Simons Fellowship.
\section{Main results}\label{sec.main}

\subsection{$q$-deformed functions}
Let us briefly recall certain $q$-deformed functions that we will use in this paper. The $q$-Pochhammer symbol is defined for $n\in \Z_{\geq 0}\cup \{+\infty\}$ as $(a;q)_n := (1-a)(1-qa)\cdots (1-q^{n-1}a)$, with the case $n=+\infty$ corresponding to the infinite convergent product. If we write $(a_1,\ldots, a_k;q)_n$ this is simply the product of $(a_1;q)_n \cdots (a_k;q)_n$. The $q$-deformed binomial coefficient is given by
$$
{n \choose j}_q := \frac{(q;q)_n}{(q;q)_{n-j}(q;q)_j}.
$$
We will work with the ${}_2 \phi_1$ basic hypergeometric function \cite{GasperRahman} which is defined as
$$
\pFq{2}{1}{a,b}{c}{q}{z}:= \sum_{n=0}^{\infty} \frac{(a;q)_n\, (b;q)_n}{(c;q)_n\, (q;q)_n} z^n.
$$

There are other $q$-deformed functions (such as the continuous $q^{-1}$-Hermite polynomials) which we will introduce as we need them in the main text.
\subsection{Dynamic ASEP}

\begin{figure}
\begin{center}
\includegraphics[scale=.8]{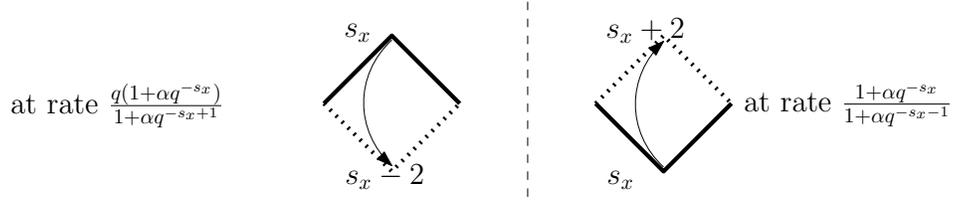}
\end{center}
\caption{Jumps of dynamic ASEP are independent exponentially distributed random variables whose rates are given above. Notice that they depend on the height function $s_x$ so that when $s_x$ gets very high, the jumps favor it decreasing, when it gets very low, the jumps favor it increasing.}\label{fig.dASEP}
\end{figure}

We start by defining the dynamic ASEP.

\begin{definition}
{\it Dynamic ASEP} is a continuous time Markov process which is defined in terms of its time $t$ height function $\vec{s}(t)\in\mathcal{S}$ where $\mathcal{S}=\Big\{\vec{s} = \big(s_{x}\big)_{x\in \Z}: \forall x\in \Z, s_x\in \Z, |s_{x+1}-s_x|=1\Big\}$. In words, $\mathcal{S}$ is the set of all height functions which take integer values at integer $x$ and have slopes $\pm 1$ in between. It remains to specify the stochastic rules for time evolution of dynamic ASEP. For all $x\in \Z$ the following jumps occur according to independent exponential clocks of specified rates (assuming that the jump does not move outside of the state-space):
$$
s_x \mapsto s_x - 2\quad \textrm{at rate}\quad \frac{q(1+\alpha q^{-s_x})}{1+\alpha q^{-s_x +1}},\qquad  s_x \mapsto s_x + 2\quad \textrm{at rate}\quad \frac{1+\alpha q^{-s_x}}{1+\alpha q^{-s_x -1}}.
$$
Here $\alpha\in \R_{\geq 0}$ is a parameter that we assume to be positive. See Figure \ref{fig.dASEP} for an illustration of this process.

The generator (see, e.g. \cite{Liggett}) of this Markov process will be denoted $\mathcal{L}_{q,\alpha}$ and is defined on local functions of the state space. We will also use the generator $L^n_{\ell,r}$ of the standard ASEP \cite{Liggett} on $n$-particle ordered integer configurations $\vec{x}= (x_1>\cdots> x_n)$ (call the set of such states $\mathcal{X}^n$) where particles jump left at rate $\ell$ and right at rate $r$, subject to the exclusion rule.
\end{definition}

\begin{remarks}\label{rem.alpha}
Notice that when $\alpha\to 0$ the downward jump rate becomes $q$ and upward becomes $1$. When $\alpha\to \infty$, the opposite happens as the downward rate becomes $1$ and the upward becomes $q$. In both cases, one recovers standard ASEP by associating slope $-1$ increments with particles, and $+1$ with holes.
\end{remarks}

\subsection{Markov duality}

Our first main result is a Markov duality of dynamic ASEP and the $n$-particle ($n\geq 1$ is arbitrary) standard ASEP with respect to the duality function $Z_{n;q,\alpha}:\mathcal{X}^n \times \mathcal{S} \to \R$ given by
\begin{equation}\label{eq.Zs}
Z_{n;q,\alpha}(\vec{x};\vec{s}) = \prod_{k=1}^{n} \Big(q^{-x_{k}} -\alpha^{-1}q^{2(k-1)} - q^{k-1}\big(q^{\frac{-s_{x_{k}}-x_{k}}{2}}-\alpha^{-1} q^{\frac{s_{x_{k}}-x_{k}}{2}}\big)\Big).
\end{equation}

\begin{theorem}\label{thm.duality}
For all integer $n\geq 1$, $q\in (0,1)$, $\alpha\in(0,\infty)$ and states $\vec{x}\in \mathcal{X}^n$ and $\vec{s}\in \mathcal{S}$,
\begin{equation}\label{eq.markdual}
\mathcal{L}_{q,\alpha} Z_{n;q,\alpha}(\vec{x};\vec{s})=L^{n}_{1,q} Z_{n;q,\alpha}(\vec{x};\vec{s})
\end{equation}
where on the left-hand side the operator acts in the $\vec{s}$ variable, while on the right-hand side the operator acts on the $\vec{x}$ variable. Consequently, for $\vec{x}\in \mathcal{X}^n$,
\begin{equation}\label{eq.evoleq}
\frac{\partial}{\partial t} \EE\big[ Z_{n;q,\alpha}(\vec{x};\vec{s}(t))\big] = L^{n}_{1,q} \EE\big[ Z_{n;q,\alpha}(\vec{x};\vec{s}(t))\big]
\end{equation}
where the expectation $\EE$ is over the time evolution of the dynamic ASEP $\vec{s}(t)$.
\end{theorem}

This result is proved in Section \ref{sec.thmproof}.

\begin{remarks}\label{rem.alphadual}
In the limit $\alpha\to 0$, up to dividing by constants (which do not affect the duality) we find, in light of Remark \ref{rem.alpha}, that standard ASEP with left jump rate $q$ and right jump rate $1$, written in terms of variables $N_x = (s_x-x)/2$, is dual to the same ASEP with $n$-particles written in terms of particle locations $\vec{x}\in \mathcal{X}^n$ with respect to the duality functional
$$
Z(\vec{x};\vec{N}) = \prod_{k=1}^{n} (q^{k-1} - q^{N_{x_{k}}}).
$$
This is not exactly the same duality functional as in \cite[Theorem 4.2]{BCS} which does not have the $q^{k-1}$ factors (i.e. there it is just $\prod_{k=1}^{n} q^{N_{x_{k}}}$). The $\alpha\to \infty$ limit also produces a duality, but with a simple conjugation it is seen to be equivalent to the $\alpha\to 0$ one.
\end{remarks}

\subsection{Initial data evaluation formulas and continuous $q^{-1}$-Hermite polynomials}

Our second set of results pertains to using the duality in Theorem \ref{thm.duality} to compute formulas for expectations (with respect to the $\vec{s}(t)$ process evolution) of $Z_{n;q,\alpha}(\vec{x};\vec{s}(t))$ for step initial data and half-stationary initial data for $\vec{s}(t)$. While step initial data is defined as usual for ASEP (see, e.g., \cite{ICReview}), half-stationary is more involved and is introduced below. In Remark \ref{sec.origins} we describe how this initial data arises by specializing certain results in \cite{IRF}.

\begin{definition}\label{def.halfstat}
{\it Step initial data} means that $s_x=s_x(0) = |x|$. Half-stationary initial data means that $s_{x}=x$ for $x\geq 1$, while for $x\leq 0$, $s_x$ is defined according to a Markov chain under which
$$
s_{x-1} = s_{x} +1 \textrm{ with probability } \frac{q^{s_x}}{\alpha + q^{s_x}}, \textrm{  and }
s_{x-1} = s_{x} -1 \textrm{ with probability } \frac{\alpha}{\alpha + q^{s_x}}.
$$
Define the generator of the half-stationary initial data Markov process to be the operator $K$ acting on functions of $f:\Z\to \R$ by
\begin{equation}\label{eq.keq}
(Kf)(s) = \frac{q^{s}}{\alpha + q^{s}} f(s+1) + \frac{\alpha}{\alpha + q^{s}} f(s-1).
\end{equation}
Clearly, for $x\leq 0$, $\EE[f(s_{x-1})|s_{x}=s] = (Kf)(s)$.
\end{definition}

%\begin{remark}
%When degenerated to standard ASEP (as in the manner of Remark \ref{rem.alphadual}), the half-stationary initial data converges to either $s_x = |1-x|+1$ or $s_x = x$ (depending on whether one takes the $\alpha \to 0$ or $\alpha\to \infty$ limit). For standard ASEP, the translation invariant ergodic stationary measures are Bernoulli product measure for arbitrary parameter $\rho\in [0,1]$. We do not presently know how to produce initial data for dynamic ASEP which admits explicit $\alpha$-independent evaluation formulas (as in Theorem \ref{thm.initialdata} below) and which limit to the $\rho\in (0,1)$ Bernoulli measures. See Section \ref{sec.stat} for a discussion on the stationary measures for dynamics ASEP.
%\end{remark}

\begin{remarks}\label{sec.origins}
As was shown in \cite{AB} (and exploited further in \cite{AA}), the half-stationary initial data for the usual ASEP, as well as a certain integrable multi-parameter generalization, can be obtained through a limit transition from the stochastic higher spin six vertex in a quadrant. The nontrivial part of these initial data comes from tuning the vertex model's inhomogeneities in the first few columns of the quadrant in a special way.

Quite similarly, the half-stationary initial data of Definition \ref{def.halfstat} can be obtained through a limit of the stochastic IRF model in the quadrant considered in \cite{IRF} (the degeneration of that IRF model to the dynamic ASEP is discussed in \cite[Section 9.4]{IRF}). More exactly, the pair of complementary probabilities $b_k^{\textrm{stoch}}$ and $d_k^{\textrm{stoch}}$ of \cite[eq.(1.1)]{IRF} converge to the transition probabilities of the Markov chain from Definition \ref{def.halfstat}
as $e^{-\I\pi(z-w+(\Lambda+1)\eta)}\to 0$ and $-e^{2\I\pi\lambda-4\I\pi\Lambda\eta}\to\alpha$, where $\I=\sqrt{-1}$. In order to turn this into the half-stationary initial data for dynamic ASEP, one needs to perform such a limit transition in the first column of the stochastic IRF model in the quadrant, while the IRF parameters in the rest of the quadrant should follow the degeneration path of \cite[Section 9.4]{IRF}. Correspondingly, the integral representation of Corollary \ref{cor.int} below can be viewed as a limit of \cite[Theorem 1.1]{IRF}. Note that in Corollary \ref{cor.int}, we are limited to moments with different $x_i$, while the IRF model methods described above would likely show that those formulas extend to the Weyl chamber boundary when multiple $x_i$ coincide.
\end{remarks}

The following theorem evaluates the duality functional for step and half-stationary initial data. It should be noted that though the duality functional (and in the half-stationary case, also the initial data) depends on $\alpha$, the resulting evaluation is $\alpha$-independent. In Section \ref{sec.intform} we use this $\alpha$-independence to write down integral formulas for expectations of dynamic ASEP evaluated through the duality functional at later times.

\begin{theorem}\label{thm.initialdata}
For step initial data $\vec{s}_{\textrm{step}}$, $n\geq 1$ and $\vec{x}\in \mathcal{X}^n$,
\begin{equation}\label{eq.alphaindstep}
(-\alpha^{-1};q)_n \,Z_{n;q,\alpha}(\vec{x};\vec{s}_{\textrm{step}}) = \prod_{k=1}^{n}\big(q^{-x_{k}\mathbf{1}_{x_{k}\leq 0}}-q^{k-1}\big).
\end{equation}
For (random) half-stationary initial data $\vec{s}_{\textrm{half}}$, $n\geq 1$ and $\vec{x}\in \mathcal{X}^n$,
\begin{equation}\label{eq.alphaindhalf}
\frac{\alpha^n}{q^{n(n-1)/2}} \, \EE\big[Z_{n;q,\alpha}(\vec{x};\vec{s}_{\textrm{half}})\big] = \prod_{k=1}^{n-1}\big(q^{(1-x_{k})\mathbf{1}_{x_{k}\leq 1 }}-q^{k-1}\big),
\end{equation}
where the expectation $\EE$ is over the randomness of $\vec{s}_{\textrm{half}}$.
\end{theorem}

We give a brief proof of \eqref{eq.alphaindstep} here since it is quite simple and then sketch the main points in the proof of \eqref{eq.alphaindhalf}, which is presented completely in Section \ref{sec.half}.

\begin{proof}[Proof of \eqref{eq.alphaindstep}]
For $s_x=|x|$ for all $x\in \Z$ we can factor $Z_{n;q,\alpha}(\vec{x};\vec{s})$ into the product
$$\prod_{k=1}^{n} \big(q^{-x_{k} \mathbf{1}_{x_{k}\leq 0}}-q^{k-1}\big)\big(q^{-x_{k}\mathbf{1}_{x_{k}\geq 0}}+\alpha^{-1} q^{k-1}\big).$$
If $x_1\geq 0$ this product is clearly 0, and since $x_n<\cdots<x_1$, this means that $Z_{n;q,\alpha}(\vec{x};\vec{s})=0$ unless all $x_i\leq 0$ (as is the case with the right-hand side of \eqref{eq.alphaindstep} as well). If all $x_i\leq 0$, the first term in the product above becomes $(q^{-x_{k}}-q^{k-1})$ and the second becomes $(1+\alpha^{-1}q^{k-1})$. The product of these second terms cancels with $(-\alpha^{-1};q)_n$ and we are left with the desired equality of \eqref{eq.alphaindstep}.
\end{proof}

We turn now to sketch the main ideas in the proof of \eqref{eq.alphaindhalf} (the complete proof is given in Section \ref{sec.half}). Our proof relies on a connection between the Markov chain which produces the half-stationary initial data, and the continuous $q^{-1}$-Hermite polynomials. These polynomials, denoted $h_{n}(x\barspace q)$, fit into the $q$-Askey-Wilson scheme and hence can be expressed in terms of basic hypergeometric series. In particular, they arise as limits of the Al-Salam-Chihara polynomials. Relying on \cite[Section 2]{CK} we will define the $h_n$ via the three term recursion that they satisfy.

\begin{definition}
The {\it continuous $q^{-1}$-Hermite orthogonal polynomials} $\big\{h_n(x\barspace q)\big\}_{n\geq 0}$ are given by the following three-term recursion:
\begin{equation}\label{eq.threeterm}
2x h_n(x\barspace q) = h_{n+1}(x\barspace q) + (q^{-n}-1) h_{n-1}(x\barspace q), \quad \textrm{with} \quad h_{-1}(x\barspace q)=0,\quad \textrm{and}\quad h_{0}(x\barspace q)=1.
\end{equation}
\end{definition}

Recall the operator $K$ from Definition \ref{def.halfstat}, and define the function
\begin{equation}\label{eq.f}
\f(s) = q^{-s/2} \alpha^{1/2} - q^{s/2}\alpha^{-1/2}, \qquad s\in \Z.
\end{equation}
We have the following lemma which shows that the functions $s\mapsto h_{n}\big(\f(s)/2\barspace  q\big)$ diagonalize $K$. This lemma is proved in Section \ref{sec.half}.

\begin{lemma}\label{lem.eign}
For all integer $n\geq 0$,
\begin{equation}
Kh_n\big(\f(s)/2\barspace q\big) = q^{n/2} h_n\big(\f(s)/2\barspace q\big).
\end{equation}
The left-hand side should be interpreted as $(Kg)(s)$ where $g(s)= h_n\big(\f(s)/2\barspace q\big)$.
\end{lemma}

To prove \eqref{eq.alphaindhalf} we observe that the right-hand side is a multinomial in the variables $q^{-x_{k}}$ for $k=1,\ldots, n$ which has at most degree one in each variable and vanishes when $x_{k}= 2-k$ for any $k\in \{1,\ldots, n\}$ (we are assuming all $x_i\leq 1$ since otherwise both sides are easily seen to be zero as in the proof of \eqref{eq.alphaindstep}) and which has coefficient $q^{-n}$ in its top degree term $\prod_{k=1}^{n-1} q^{-x_{k}}$. This characterizes the right-hand side completely, hence it suffices to check that the same characterization holds on the left-hand side. The vanishing of the left-hand side when $x_{k}= 2-k$ for any $k\in \{1,\ldots, n\}$ follows easily by looking at the range of $s_x$. What remains is to show that the left-hand side is a multinomial which has at most degree one in each term and that it has the desired top degree coefficient. Since the operator $K$ computes expectations of functions under the trajectory of $s_x$, we can use the eigenrelation in Lemma \ref{lem.eign} to extract powers of $q^{x_{k}}$ and confirm the multinomiality and degree conditions. The same sort of considerations produce a formula for the maximal degree coefficient, expressed in terms of sums of weighted $h_n$. Establishing that this evaluates to $q^{-n}$ boils down to the following identity which is proved in Section \ref{sec.half}.

\begin{lemma}\label{sumid}
For all integers $n\geq 0$ and any non-zero $\alpha$ and $q$
\begin{equation}\label{eq.sumid}
\frac{\alpha^n}{q^{n(n+1)/2}} \sum_{j=0}^{n} (-1)^j {n \choose j}_q q^{j(j+1)/2} (q\alpha)^{-j/2} h_j\big(\f(1)/2\barspace q\big) = 1.
\end{equation}
\end{lemma}

%\begin{remark}
%Though we do not use it, let us record the transition probabilities for the half-stationary initial data Markov process. Let $S(t) = s_{1-t}$. Then for $t\in \Z_{\geq 0}$, $x\in \Z$ and $y\in \{0,\ldots t\}$,
%$$
%\PP\big(S(t)=x+2y-t\big\vert S(0)x\big) = {t \choose y}_q \alpha^{-y}q^{(x-t)y + \frac{y(3y-1)}{2}} \,\frac{(1+\alpha^{-1} q^{x+2y-t})(-\alpha^{-1} q^{1+x-t};q)_{y-1}}{(-\alpha^{-1} q^{1+x};q)_y (-\alpha^{-1} q^{1+x-t};q)_t}.
%$$
%\end{remark}

%****Half****
%The independence of the initial data of alpha is more involved and relies on two main facts 1. one-step initial data markov generator is diagonalized by $q^{-1}$-Hermite's and 2. give the summation formula. This relation to q-orthogonal polynomials is quite compelling and deserves investigation for the dynamic hs6v type models and for more general integrable initial data -- the IRF methods also produce this final alpha independence equality, hence they implicitly encode some of these relations.
%************

\subsection{Integral formulas}\label{sec.intform}

As a corollary of Theorems \ref{thm.duality} and \ref{thm.initialdata} we find that expectations of $Z_{n;q,\alpha}$ solve the ASEP evolution equation with $\alpha$-independent initial data. We are able to write down (and readily confirm) explicit integral solutions to these equations with the specified initial data. One way to produce such formulas is through the spectral theory for ASEP developed in \cite{BCPS}.

\begin{corollary}\label{cor.int}
For step initial data $\vec{s}(0)=\vec{s}_{\textrm{step}}$, $n\geq 1$ and $\vec{x}\in \mathcal{X}^n$,
$$
E_{\textrm{step}}(t;\vec{x}) := (-\alpha^{-1};q)_n \EE\big[\,Z_{n;q,\alpha}(\vec{x};\vec{s}(t))\big]
$$
solves the system of ODEs
$$
\frac{d}{dt} E_{\textrm{step}}(t;\vec{x}) = L^{n}_{1,q} E_{\textrm{step}}(t;\vec{x}), \quad \textrm{with} \quad E_{\textrm{step}}(0;\vec{x}) =  \prod_{k=1}^{n}\big(q^{-x_{k}\mathbf{1}_{x_{k}\leq 0}}-q^{k-1}\big).
$$
This equation admits the following explicit integral solution
$$
E_{\textrm{step}}(t;\vec{x}) = \frac{q^{n(n-1)/2}}{(2\pi \i)^n} \oint \cdots \oint \prod_{1\leq i<j\leq n} \frac{y_i -y_j}{y_i-q y_j} \prod_{i=1}^{n} \left(\left(\frac{1-y_i}{1-qy_i}\right)^{x_i} \exp\left\{\frac{(1-q)^2 y_i}{(1-y_i)(1-q y_i)}\, t\right\}\right) \frac{dy_i}{y_i},
$$
with integration contours given by small positively oriented loops around $1$.

For half-stationary initial data $\vec{s}(0)=\vec{s}_{\textrm{half}}$, $n\geq 1$ and $\vec{x}\in \mathcal{X}^n$,
$$
E_{\textrm{half}}(t;\vec{x}) := \frac{\alpha^n}{q^{n(n-1)/2}} \EE\big[\,Z_{n;q,\alpha}(\vec{x};\vec{s}(t))\big]
$$
solves the system of ODEs
$$
\frac{d}{dt} E_{\textrm{half}}(t;\vec{x}) = L^{n}_{1,q} E_{\textrm{half}}(t;\vec{x}), \quad \textrm{with} \quad E_{\textrm{half}}(0;\vec{x}) =  \prod_{k=1}^{n}\big(q^{(1-x_{k})\mathbf{1}_{x_{k}\leq 1}}-q^{k-1}\big).
$$
Since $E_{\textrm{half}}(0;\vec{x}) = E_{\textrm{step}}(0;\vec{x}-\vec{1})$ (with $\vec{1}$ a vector of length $n$ with all $1$'s), it immediately follows that $E_{\textrm{half}}(t;\vec{x})= E_{\textrm{step}}(t;\vec{x}-\vec{1})$ as well.
\end{corollary}

%
%\begin{remark}
%Corollary \ref{cor.int} gives formulas for $\EE\big[\,Z_{n;q,\alpha}(\vec{x};\vec{s}(t))\big]$ when $\vec{x}\in \mathcal{X}^n$. This means the $x_i$ are strictly ordered. Using the methods of \cite{IRF} (as described briefly in Remark \ref{sec.origins}) it is possible to prove these formulas extending to the boundary where the $x_i$ are weakly ordered (allowing for equal consecutive $x_i$).
%\end{remark}

\begin{proof}
Theorems \ref{thm.duality} and \ref{thm.initialdata} immediately imply the evolution equations and initial data. It is easy to check that the integrals are the unique solutions to these equation. To solve $\frac{d}{dt}E(t;\vec{x}) = L^{n}_{\ell,r} E(t;\vec{x})$ it suffices (see \cite[Proposition 4.11]{BCS}) to find a function $u(t;\vec{x})$ with $\vec{x}\in \Z^n$ which satisfies (1) the {\it free evolution equation} $\frac{d}{dt}u(t;\vec{x}) = \sum_{i=1}^{n} \big[L^{1}_{\ell,r}\big]_i u(t;\vec{x})$ (where the subscript $i$ means that the operator acts only on $x_i$), (2) the {\it two-body boundary conditions} $(\ell \nabla_{i}^{-} - r \nabla_{i+1}^{+}) u(\vec{x}) =0$ for all $\vec{x}$ with $x_{i}=x_{i+1}+1$ and all $i\in \{1,\ldots, n-1\}$ (where $(\nabla^{\pm}f)((x) = f(x\pm 1)-f(x)$ and the subscript $i$ indicates to act on $x_i$), (3) the initial data $u(0;\vec{x}) = E(0;\vec{x})$ for $\vec{x}\in \mathcal{X}^n$. In fact, as explained in \cite[Proposition 4.11]{BCS}, one must really verify an exponential bound on $u(t;\vec{x})$ for all $t$ and point-wise convergence to the initial data as $t\to 0$. The verification that our integral formula satisfies these points is by this point quite standard and one can find an example of a detailed treatment in, for example, \cite[Theorem 4.13]{BCS}. Briefly, the free evolution equation is verified by combining
$$
\frac{d}{dt} \left(\frac{1-y}{1-qy}\right)^{x} \exp\left\{\frac{(1-q)^2 y}{(1-y)(1-q y)}\, t\right\} = L^1_{1,q} \left(\frac{1-y}{1-qy}\right)^{x} \exp\left\{\frac{(1-q)^2 y}{(1-y)(1-q y)}\, t\right\}
$$
with Leibnitz's rule. For any $i\in \{1,\ldots, n-1\}$, applying the boundary condition operator $\nabla_{i}^{-} - q \nabla_{i+1}^{+}$ when $x_i=x_{i+1}+1$ to the integrand results in a new integrand which is anti-symmetric in $y_i$ and $y_{i+1}$ (the denominator $(y_i-qy_{i+1})$ clears and the only asymmetric term left is $(y_i-y_{i+1})$). Since all integrals are on the same contours, this total integral is necessarily zero, verifying the boundary condition. For the initial data, if $x_1\geq 0$, there is no residue in $y_1$ at $1$ and the integral can be evaluated to zero. Otherwise, if all $x_i<0$, there are simple poles at $\infty$. One progressively deforms the contours there, but in so doing crosses residues from the denominator $(y_i-qy_j)$. This results in a residue expansion which is readily matched to the desired factor.
\end{proof}

\subsection{Stationary initial data and moment indeterminacy}\label{sec.stat}

We start by defining what we call stationary initial data. This is stationary in two senses. The first is that it is the trajectory of a stationary (up to parity) Markov chain in $x$. The second is that this measure on height functions is stationary for the dynamic ASEP time evolution.

\begin{definition}\label{def.stat}
We say that $\vec{s}$ is distributed as stationary initial data if $\big\{s_x\big\}_{x\in \Z}$ is defined as the trajectory of a Markov process with the same transition probabilities from $s_x$ to $s_{x-1}$ as the half-stationary initial data from Definition \ref{def.halfstat}, and with the one-point marginal distribution
$$
\PP(s_0 = 2n) = m_n,\qquad \textrm{where}\quad m_n = \frac{\alpha^{-2n} q^{n(2n-1)}(1+\alpha^{-1} q^{2n})}{(-\alpha^{-1},-q\alpha,q;q)_\infty}.
$$
\end{definition}
\begin{remarks}\label{rem.eqnfo}
It is not immediately obvious that $m_n$ defines a probability measure on $\Z$. This can be shown from the Jacobi triple product formula, and also follows from the following continuous $q^{-1}$-Hermite polynomial orthogonality statement which appears as \cite[Equation (4.1)]{CK}:
$$
\sum_{n\in \Z} m_n \,h_{a}\big(\f(n)/2\barspace q\big)h_{b}\big(\f(n)/2\barspace q\big) = \delta_{a,b} \,\frac{(q;q)_a}{q^{a(a+1)/2}}.
$$
Here $a,b \in \Z_{\geq 0}$. Taking $a=b=0$ so the $h_0\equiv 1$, yields that $m_n$ is a probability measure. The $h_a$ are orthogonal polynomials for this measure.
\end{remarks}

The next lemma shows that the stationary initial data is (up to parity) stationary (in terms of its one-point marginals) in space.
\begin{lemma}\label{lem.statx}
The Markov chain $s_x$ from Definition \ref{def.stat} has marginal distributions such that for $k\in \Z$,
$$
\PP(s_{2k} = 2n) = m_n,\qquad \textrm{and}\qquad \PP(s_{2k+1} = 2n+1) = \tilde{m}_n
$$
where $\tilde{m}_n$ is the same as $m_n$ but with $\alpha$ replaced with $\alpha/q$.
\end{lemma}
\begin{proof}
This follows immediately from the readily checked fact that
$$
\tilde{m}_n = \frac{q^{2n}}{\alpha + q^{2n}} \, m_n + \frac{\alpha}{\alpha+q^{2(n+1)}}\, m_{n+1}\qquad \textrm{and}\qquad
m_n = \frac{q^{2n-1}}{\alpha + q^{2n-1}} \, \tilde{m}_{n-1} + \frac{\alpha}{\alpha+q^{2n+1}}\, \tilde{m}_{n}.
$$
\end{proof}

This initial data (as a height function) is also stationary with respect to the dynamic ASEP evolution.

\begin{theorem}\label{thm.stationary}
Stationary initial data (Definition \ref{def.stat}) is stationary for dynamic ASEP.
\end{theorem}
\begin{proof}
This follows immediately by checking that the stationary initial data respects detailed balance for the jump rates of dynamic ASEP. In particular, for any $x$ and any $s$ (with the same parity as $x$) we calculate
\begin{equation}\label{eqnabv}
\frac{\textrm{rate}(s-1 \mapsto s+1)}{\textrm{rate}(s+1 \mapsto s-1)} = \frac{1+\alpha q^{-s+1}}{1+\alpha q^{-s}} \, \cdot \, \frac{1+\alpha q^{-s}}{q(1+\alpha q^{-s-1})} =  \frac{q^{-1}(1+\alpha q^{-s+1})}{1+\alpha q^{-s-1}}.
\end{equation}
Under the stationary measure, the probability of seeing the height function pattern $\wedge$ (starting and ending at height $s$) is $q^s(\alpha+q^s)^{-1} \alpha (\alpha+q^{s+1})$  while the probability of seeing the pattern $\vee$ (also starting and ending at height $s$) is $\alpha(\alpha+q^s)^{-1} q^{s-1}(\alpha+q^{s-1})^{-1}$. The ratio of these probabilities is equal to \eqref{eqnabv} and hence detailed balance is satisfied. This implies the stationarity of the initial data (cf. \cite[Proof of Theorem 2.1]{Liggett} for details on how detailed balance implies stationarity).
\end{proof}

\begin{remarks}\label{remonep}
As noted after \cite[Definition 9.2]{IRF}, the dynamic ASEP has a slightly simpler  description if one replaces $s_x$ by $\tilde{s}_x = s_x - \log_q \alpha$ (for simplicity, assume $\alpha>0$ here). The new height function $\tilde{s}$ now lies on a shift of the integer lattice, but maintains the $\pm 1$ slope line increments. The jump rates (assuming that they can be made) are now $\alpha$-independent, and $\tilde{s}_x\mapsto \tilde{s}_x -2$ with rate $\frac{q(1+q^{-\tilde{s}_x})}{1+q^{-\tilde{s}x+1}}$ and $\tilde{s}_x\mapsto \tilde{s}_x+2$ with rate $\frac{1+q^{-\tilde{s}_x}}{1+q^{-\tilde{s}_x-1}}$. Theorem \ref{thm.stationary} provides a one-parameter family of stationary initial data for this modified version of the dynamic ASEP. The parameter is the value of $\log_q\alpha$ and it is allowed to vary in the interval $[0,1)$.
\end{remarks}

\begin{remarks}
The one parameter family of stationary measures discussed in Remark \ref{remonep} are related (as we now explain) to the one-parameter family of measures which satisfy the indeterminant moment problem with respect to the orthogonality measure of the continuous $q^{-1}$-Hermite polynomials \cite{ISV}. Consider the dynamic ASEP with step initial data so that $s_x(0)=|x|$. Inspecting the jump rates, one expects that after a long time, above the origin the height function should be roughly around height $A$ with $q^A=\alpha$ (this is the height above which jumps down become more likely than jumps up, and below which the opposite occurs). It is, in fact, possible to compute certain expectations of this long time height distribution using our duality result. Theorem \ref{thm.duality} implies that $$\EE_{\vec{s}}\big[Z_{n;q,\alpha}(\vec{x};\vec{s}(t))\big] = \EE_{\vec{x}}\big[Z_{n;q,\alpha}(\vec{x}(t);\vec{s})\big]$$
where the first expectation is over the $\vec{s}(t)$ evolution from initial data $\vec{s}$ and the second is over the $\vec{x}(t)$ evolution (with left jump rate $1$ and right jump rate $q$) with initial data $\vec{x}$. We know, from Theorem \ref{thm.initialdata}, how to evaluate the right-hand side above when $\vec{s}$ is step initial data:
$$
(-\alpha^{-1};q)_n \EE_{\vec{s}_{\textrm{step}}}\big[Z_{n;q,\alpha}(\vec{x};\vec{s}(t))\big] = \EE_{\vec{x}} \Big[\prod_{k=1}^{n} \big(q^{-x_k\mathbf{1}_{x_k\leq 0}}-q^{k-1}\big)\Big].
$$
Now take $x_k\equiv 0$ and $t\to \infty$\footnote{This is actually not allowed, as our duality requires strict ordering of the $x_k$ so one cannot take all the $x_k$ equal. On the other hand, there is another route to prove \eqref{eqn.momfo}. The methods of \cite{IRF} show that the integral formulas in Corollary \ref{cor.int} hold true even when all $x_k$ are equal. Taking asymptotics as $t\to \infty$ of those formulas yields \eqref{eqn.momfo}.}.  We know that $x_k(t)$ will all tend to $-\infty$ (since the left jump rate exceeds the right one). Thus, the expectation simplifies as $q^{-x_k}\to 0$ and we get
\begin{equation}\label{eqn.momfo}
(-\alpha^{-1};q)_n \EE_{\vec{s}_{\textrm{step}}}\Big[\prod_{k=1}^n \big(1-\alpha^{-1} q^{2(k-1)} - q^{k-1}\alpha^{-1/2} \zeta(\alpha)\big) \Big] = (-1)^n q^{n(n-1)/2}
\end{equation}
where $\zeta(\alpha) = \alpha^{1/2} q^{-s_0(\infty)/2} - \alpha^{-1/2} q^{s_0(\infty)/2}$. Here we have assumed that $s_0(t)$ has a limiting distribution $s_0(\infty)$. The above relation can be used to extract the moments of $\zeta(\alpha)$. Note that $\zeta(\alpha)$ depends on $\alpha$ not just through the factors of $\alpha^{1/2}$ but also in that the dynamics which give rise to that limiting random variable are $\alpha$-dependent.

To relate the above discussion to the stationary measure, we should take $\alpha\to 0$ (otherwise, the long time height function will only differ from $|x|$ in a finite sized neighborhood of the origin). Doing this implies that the initial data $|x|$ around the origin rises to a very high level $A$. One may then hope that the properly centered height function will converge as $\alpha\to 0$ (in the vicinity of the origin) to a stationary height function. Using \eqref{eqn.momfo} to extract moments $\EE\big[ \big(\zeta(\alpha)\big)^k\big]$ for small $k$ and then taking $\alpha\to 0$ we were able to guess that the moments of the $\alpha\to 0$ limit of $\zeta$ (assuming it exists, call it $\zeta$ without any argument) satisfy
$$
\EE[\zeta^k] = \begin{cases} (1-q^{-1})^{k/2} \sum_{i=-k/2}^{k/2}{k \choose k/2+i} (-1)^i q^{-i(i-1)/2},&k\textrm{ is even};\\ 0,&k\textrm { is odd}.\end{cases}
$$
These are exactly the moments of the continuous $q^{-1}$-Hermite polynomials orthogonality measure \cite{ISV}. Having $q^{-1}$ (with $q\in (0,1)$) makes the moment problem indeterminate, but remarkably all solutions have been found in \cite{IM}. The extreme points of the convex set of solutions are distinguished by their support (they are all discrete). The weights of those are $m_n$ (from Definition \ref{def.stat}) at the points $x_n = \alpha^{1/2}q^{-n} - \alpha^{-1/2} q^n$ for $n\in \Z$. Thus, our stationary measures (after the change of variables in Remark \ref{remonep}) are exactly in correspondence with these extreme points.
\end{remarks}

\section{Proof of Theorem \ref{thm.duality}}\label{sec.thmproof}
This proof follows a similar method as in the proof of \cite[Theorem 4.2]{BCS}.

It is convenient to associate to $\vec{s}\in \mathcal{S}$ a particle configuration $\vec{\eta}=(\eta_{x+1/2})_{x\in \Z}$ where $\eta_{x+1/2} =\frac{1+ s_{x}-s_{x+1}}{2}\in \{0,1\}$ (we think of $1$ as representing particles or negatively sloping increments and $0$ as representing holes or positively sloping increments) as well as $\vec{N} = (N_x)_{x\in \Z}$ where $N_x = \frac{s_x-x}{2}$. Note that $N_{x+1} = N_x - \eta_{x+1/2}$. All Markov processes $\vec{s}(t)$, $\vec{\eta}(t)$ and $\vec{N}(t)$ should be considered as coupled by the above equalities. We can rewrite \eqref{eq.Zs} in terms of $\vec{N}$ variables. We will, in fact, need to work with the following generalization of that definition. For any subset $I\subseteq \{1,\ldots,n\}$ let
\begin{equation}\label{eq.ZN}
Z_{I;q,\alpha}(\vec{x};\vec{N}) = \prod_{k\in I} \Big(q^{-x_{k}} -\alpha^{-1}q^{2{k-1}} - q^{k-1}\big(q^{-N_{x_{k}}-x_{k}}-\alpha^{-1} q^{N_{x_{k}}}\big)\Big).
\end{equation}
(We will suppress the $q,\alpha$ subscript on $Z$ in what follows.) For any interval $I\subseteq \{1,\ldots, n\}$ let $L^{I}_{\ell,r}$ denote the generator of $|I|$-particle ASEP acting on the variables $\{x_i\}_{i\in I}$. In this way, $Z_{[1,n];q,\alpha} = Z_{n;q,\alpha}$ and $L^{[1,n]}_{\ell,r} = L^n_{\ell,r}$, where we use the abbreviation $[a,b]=\{a,a+1,\ldots, b\}$ for integers $a<b$.

The following lemma demonstrates the duality when $\vec{x}$ has a single cluster of neighboring particles.
\begin{lemma}\label{lem.subdual}
For any subset $[a,b]\subseteq [1,n]$ and $x\in \Z$, if $x_{i} = x+a+b-i$ for $a\leq i\leq b$, then
\begin{equation}\label{eq.markdualrest}
\mathcal{L}_{q,\alpha} Z_{[a,b];q,\alpha}(\vec{x};\vec{s}) = L^{[a,b]}_{1,q} Z_{[a,b];q,\alpha}(\vec{x};\vec{s}).
\end{equation}
\end{lemma}

\begin{proof}
Rewriting both sides of \eqref{eq.markdualrest} in terms of the actions of the generators and dividing through by $Z_{[a,b];q,\alpha}(\vec{s};\vec{x})$, we find that the desired equality is equivalent to showing that
\begin{equation}\label{eq.reduction}
\sum_{k=a-1}^{b-1} L(k,n) = A(a,n)+B(b,n)
\end{equation}
where
\begin{eqnarray*}
L(k,n) &=& \eta_{x+k-1/2}(1-\eta_{x+k+1/2})\,  \frac{1+\alpha q^{-s_{x+k}}}{1+\alpha q^{-s_{x+k} -1}}\,  \left(\frac{Z_{[a,b]}(\vec{s}_{x+k}^{+2};\vec{x})}{Z_{[a,b]}(\vec{s};\vec{x})}-1\right)\\
\nonumber && + (1-\eta_{x+k-1/2})\eta_{x+k+1/2}\, \frac{q(1+\alpha q^{-s_{x+k}})}{1+\alpha q^{-s_{x+k} +1}}\,  \left(\frac{Z_{[a,b]}(\vec{s}_{x+k}^{-2};\vec{x})}{Z_{[a,b]}(\vec{s};\vec{x})}-1\right),\\
A(a,n) &=& \frac{Z_{[a,b]}(\vec{s};\vec{x}^{-1}_{a})}{Z_{[a,b]}(\vec{s};\vec{x})}-1, \\
B(b,n) &=& q \left(\frac{Z_{[a,b]}(\vec{s};\vec{x}^{+1}_{b})}{Z_{[a,b]}(\vec{s};\vec{x})}-1\right);
\end{eqnarray*}
the notation $\vec{s}_{y}^{d}$ means to replace $s_{y}$ by $s_{y}+d$, and likewise $\vec{x}_{c}^d$ means to replace $x_{c}$ by $x_{c}+d$.

We can rewrite $L,A,B$ in terms of the $N$ and $\eta$ parameters, and by canceling common terms in the ratios of $Z$'s we arrive at the expressions
\begin{eqnarray*}
L(k,n) &=& \eta_{x+k-1/2}(1-\eta_{x+k+1/2})\, \frac{1+\alpha q^{-s_{x+k}}}{1+\alpha q^{-s_{x+k} -1}}\\ && \qquad \times \left(\frac{q^{-(x+k)}-\alpha^{-1} q^{2(n-k-1)} - q^{n-k-1}(q^{-1-N_{x+k}-(x+k)}-\alpha^{-1} q^{1+N_{x+k}})}{q^{-(x+k)}-\alpha^{-1} q^{2(n-k-1)} - q^{n-k-1}(q^{-N_{x+k}-(x+k)}-\alpha^{-1} q^{N_{x+k}})}-1\right)\\
\nonumber && + (1-\eta_{x+k-1/2})\eta_{x+k+1/2}\, \frac{q(1+\alpha q^{-s_{x+k}})}{1+\alpha q^{-s_{x+k} +1}}\\&&\qquad \times \left(\frac{q^{-(x+k)}-\alpha^{-1} q^{2(n-k-1)} - q^{n-k-1}(q^{1-N_{x+k}-(x+k)}-\alpha^{-1} q^{-1+N_{x+k}})}{q^{-(x+k)}-\alpha^{-1} q^{2(n-k-1)} - q^{n-k-1}( q^{-1-N_{x+k}-(x+k)}-\alpha^{-1} q^{1+N_{x+k}})}-1\right),\\
A(a,n) &=& \frac{q^{-(x+a-2)}-\alpha^{-1} q^{2(n-a)} - q^{n-a}(q^{-1-N_{x+a-2}-(x+a-2)}-\alpha^{-1} q^{1+N_{x+a-2}})}{q^{-(x+a-1)}-\alpha^{-1} q^{2(n-a)} - q^{n-a}(q^{-1-N_{x+a-1}-(x+a-1)}-\alpha^{-1} q^{1+N_{x+a-1}})} -1, \\
B(b,n) &=& q\left(\frac{q^{-(x+b)}-\alpha^{-1} q^{2(n-b)} - q^{n-b}( q^{-1-N_{x+b}-(x+b)}-\alpha^{-1} q^{1+N_{x+b}})}{q^{-(x+b-1)}-\alpha^{-1} q^{2(n-b)} - q^{n-b}( q^{-1-N_{x+b-1}-(x+b-1)}-\alpha^{-1} q^{1+N_{x+b-1}})}-1\right).
\end{eqnarray*}

One readily confirms the following formulas
\begin{equation*}
L(m,n)  = \begin{cases}
\dfrac{(1-q)q\alpha}{q^{n+x+N_{x+m+1}}+q\alpha} - \dfrac{(1-q)q\alpha}{q^{n+x+N_{x+m}}+q\alpha} & (\eta_{x+m-1/2},\eta_{x+m+1/2})=(0,0),\\ \\
\dfrac{(1-q)q^{m+2+N_{x+m+1}}}{q^{m+2+N_{x+m+1}}-q^n} - \dfrac{(1-q)q\alpha}{q^{n+x+N_{x+m}}+q\alpha} & (\eta_{x+m-1/2},\eta_{x+m+1/2})=(0,1),\\ \\
\dfrac{(1-q)q\alpha}{q^{n+x+N_{x+m+1}}+q\alpha} - \dfrac{(1-q)q^{m+1+N_{x+m}}}{q^{m+1+N_{x+m}}-q^n} & (\eta_{x+m-1/2},\eta_{x+m+1/2})=(1,0),\\ \\
\dfrac{(1-q)q^{m+2+N_{x+m+1}}}{q^{m+2+N_{x+m+1}}-q^n} - \dfrac{(1-q)q^{m+1+N_{x+m}}}{q^{m+1+N_{x+m}}-q^n} & (\eta_{x+m-1/2},\eta_{x+m+1/2})=(1,1). %- \dfrac{(1-q)q\alpha}{q^{n+x+N_{x+m}}+q\alpha}
\end{cases}
\end{equation*}
Similarly, we see that
\begin{equation*}
A(a,n)  = \begin{cases}
-\dfrac{(1-q)q\alpha}{q^{n+x+N_{x+a-1}}+q\alpha} & \eta_{x+a-3/2}=0,\\ \\
-\dfrac{(1-q)q^{a+N_{x+a-1}}}{q^{a+N_{x+a-1}}-q^n} & \eta_{x+a-3/2}=1.
\end{cases}
\end{equation*}
and
\begin{equation*}
B(b,n)  = \begin{cases}
\dfrac{(1-q)q\alpha}{q^{n+x+N_{x+b}}+q\alpha} & \eta_{x+b-1/2}=0,\\ \\
\dfrac{(1-q)q^{b+1+N_{x+b}}}{q^{b+1+N_{x+b}}-q^n}& \eta_{x+b-1/2}=1.
\end{cases}
\end{equation*}

From these evaluation formulas, \eqref{eq.reduction} follows by telescoping of the $L(m,n)$ summands. Alternatively, this can be seen by induction in $b$. When $b=a$, by inspection $L(a-1,n)= A(a,n)+B(a,n)$. The inductive step follows since, by inspection, one sees that $L(b,n)+B(b,n) = B(b+1,n)$.
\end{proof}

We now return to proving \eqref{eq.markdual}. Any $x\in \mathcal{X}^n$ can be decomposed into clusters of consecutive particles with each cluster separated by a positive number of holes (e.g. $\vec{x} = (5,4,3,1,0,-4,-5)$ has three clusters, the first $(5,4,3)$, the second $(1,0)$ and the third $(-4,-5)$). Let $c$ be the number of clusters in $\vec{x}$ and $I_i$ be the set of labels of the $i^{th}$ cluster for $1\leq i \leq c$ (e.g., $c=3$ in the above example and $I_1 = \{1,2,3\}$, $I_2=\{4,5\}$ and $I_3=\{6,7\}$). Then, letting $L^{I_i}_{1,q}$ act on the variables with labels in $I_i$ as $L^{|I_i|}_{1,q}$ we have the following string of equalities (we again suppress $q,\alpha$ in the $Z$ notation as they are fixed)
\begin{eqnarray*}
L^{n}_{1,q} Z_{[1,n]}(\vec{x};\vec{s})  &=& \sum_{i=1}^{c} L^{I_i}_{1,q} Z_{[1,n]}(\vec{x};\vec{s}) = \sum_{i=1}^{c} Z_{[1,n]\setminus I_i}(\vec{x};\vec{s}) L^{I_i}_{1,q} Z_{I_i}(\vec{x};\vec{s})\\
&=& \sum_{i=1}^{c} Z_{[1,n]\setminus I_i}(\vec{x};\vec{s}) \mathcal{L}_{q,\alpha} Z_{I_i}(\vec{x};\vec{s}) = \mathcal{L}_{q,\alpha} Z_{[1,n]}(\vec{x};\vec{s}).
\end{eqnarray*}
The first equality comes from the fact that for ASEP, disjoint clusters do not interact instantaneously. The second equality follows because $L^{I_i}_{1,q}$ only acts in the variables $\{x_j\}_{j\in I_i}$ and hence we can factor out the multiplicative terms not involving these variables. The third equality follows from applying Lemma \ref{lem.subdual}. The final equality follows because the jumps of $\vec{s}$ which affect the value of $Z_{I_i}(\vec{x};\vec{s})$ do not affect the value of $Z_{[1,n]\setminus I_i}(\vec{x};\vec{s})$. This completes the proof of \eqref{eq.markdual}. The second claimed result \eqref{eq.evoleq} follows immediately from \eqref{eq.markdual}. The proof of Theorem \ref{thm.duality} is complete.

\section{Proof of Theorem \ref{thm.initialdata}, and Lemmas \ref{lem.eign} and \ref{sumid}}\label{sec.half}

We first prove Theorem \ref{thm.initialdata}, assuming the validity of  Lemmas \ref{lem.eign} and \ref{sumid}. Then, we prove those lemmas.

\begin{proof}[Proof of Theorem \ref{thm.initialdata}]

Theorem \ref{thm.initialdata} made two claims -- equations \eqref{eq.alphaindstep} and \eqref{eq.alphaindhalf}. The first of these equations was proved immediately after the statement of the theorem, so this proof is concerned only with proving \eqref{eq.alphaindhalf}. We will prove the more general claim with weakly ordered $x_i$'s: $x_1\geq \cdots \geq x_n$.

\smallskip
\noindent {\it Case when $q\to 1$:} As it is informative for the general $q$ proof, let us first consider the proof of \eqref{eq.alphaindhalf} when $q\to 1$. In that case, the half-stationary initial data from Definition \ref{def.halfstat} has transition probabilities so that $s_{x-1}=s_x+1$ with probability $(\alpha+1)^{-1}$ and $s_{x-1}=s_x-1$ with probability $\alpha(\alpha+1)^{-1}$. The desired relation  \eqref{eq.alphaindhalf} reduces to
\begin{equation}\label{eq.qoneres}
\frac{(1+\alpha)^n}{2^n}\EE \Big[ \prod_{k=1}^{n}(2(k-1)-s_{x_{k}}+x_{k})\Big] = \prod_{k=1}^{n} (x_{k}+k-2).
\end{equation}

Making the change of variables $t_k = 1-x_k$ for $k=1,\ldots,n$ and setting $y(t) = (s_{1-t}+t-1)/2$ we find that for $t\leq 0$, $y(t)=0$ and for $t\geq 1$, $y(t)= z_1+\cdots +z_t$ where the $z_i$ are independent identically distributed Bernoulli random variables which equal $1$ with probability $(\alpha+1)^{-1}$ and 0 with probability $\alpha(\alpha+1)^{-1}$. In these variables, showing \eqref{eq.qoneres} is equivalent to showing that for all integers $t_1\leq \cdots \leq t_n$,
\begin{equation}\label{eq.qoneresred}
(1+\alpha)^n \EE\Big[\prod_{k=1}^{n} \big(y(t_{k})-k+1\big)\Big] = \prod_{k=1}^{n} (t_{k}-k+1).
\end{equation}

Observe that the right-hand side of \eqref{eq.qoneresred} is a multinomial in the variables $\{t_1,\ldots, t_n\}$ with maximal degree 1 in each $t_i$ term. The multinomial vanishes when $t_{k}=k-1$ for some $k\in \{1,\ldots, n\}$ and has coefficient 1 in front of the maximal total degree term $t_1\cdots t_n$. This is a complete characterization of the multinomial, hence it suffices to show that the left-hand side of \eqref{eq.qoneresred} likewise satisfies these properties.

Let us first demonstrate that the left-hand side of \eqref{eq.qoneresred} is a multinomial in the variables $\{t_1,\ldots, t_n\}$ with maximal degree 1 in each $t_i$ term. To do so, it suffices to show the same for $\EE\big[\prod_{k\in I} y(t_{k})\big]$ for any $I\subseteq \{1,\ldots, n\}$. Since $y(t) = z_1+\cdots +z_t$ we can reduce the computation of that expectation to the sum of expectations of $\prod_{k\in I} z_{s_k}$ where $s_k$ ranges over $\{1,\ldots, t_k\}$. The expectation of such a product is determined by the number $d$ of distinct $s_k$ for $k\in I$ and the resulting expectation equals $(\alpha+1)^{-d}$. Given a collection of sets $\{1,\ldots, t_k\}$ for $k\in I$, the number of ways of picking one element $s_k$ from each set so as to have $d$ district elements in $\{s_k\}_{k\in I}$ is easily seen to be a multinomial in variables $\{t_k\}_{k\in I}$ of maximal degree 1 in each variable.  This implies the desired multinomiality of the left-hand side of \eqref{eq.qoneresred}. Moreover, it shows that the maximal degree term $t_1\cdots t_n$ only arises when $I= \{1,\ldots,n\}$ and $d=n$. In that case, it is a easy to see that the $t_1\cdots t_n$-coefficient in $\EE\big[\prod_{k=1}^n y(t_{k})\big]$ is $(\alpha+1)^{-n}$. This cancels the factor $(1+\alpha)^n$ in front of the expectation in the left-hand side of \eqref{eq.qoneresred} and gives the desired maximal coefficient of 1.

\begin{figure}
\begin{center}
\includegraphics[scale=1.2]{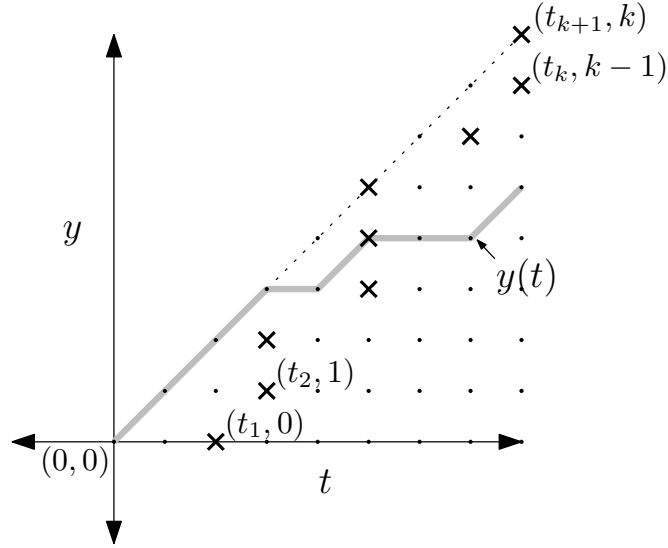}
\end{center}
\caption{Proof that the left-hand side of \eqref{eq.qoneresred} vanishes when $t_{k+1}=k$.}\label{fig.tygraph}
\end{figure}

It remains to demonstrate that the left-hand side of \eqref{eq.qoneresred} vanishes when $t_{k+1}=k$ for some $k\in \{0,\ldots, n-1\}$. This fact is illustrated in Figure \ref{fig.tygraph}. Assume that $t_{k+1}=k$. On $(t,y)$-axes we label with $\mathbf{\times}$ points $(t_{\ell+1},\ell)$ for $0\leq \ell\leq k$. In bold grey we plot the trajectory $\big(t,y(t)\big)$ for $0\leq t\leq k$. Notice that since $t_1\leq \cdots \leq t_{k+1}$, the grey trajectory (which starts at $(0,0)$) must necessarily hit at least one $\mathbf{\times}$ labeled point. That occurrence, however, implies that $y(t_{\ell+1}) - \ell=0$, and hence the product $\prod_{k=0}^{n-1} \big(y(t_{k+1})-k\big)$ for that trajectory $y(\cdot)$ will be zero. Since this holds for all trajectories $y(\cdot)$ on which $\EE$ is supported, the expectation must likewise vanish. This argument completes the proof \eqref{eq.qoneresred} and \eqref{eq.qoneres} -- the $q\to 1$ limit of our desired result  \eqref{eq.alphaindhalf}.

\smallskip
\noindent {\it General $q$ case:} We return now to prove  \eqref{eq.alphaindhalf} for general $q$. For convenience, let us restate the desired result below
\begin{equation}\label{eq.alphaindhalfrestate}
\frac{\alpha^n}{q^{n(n-1)/2}} \, \EE\big[Z_{n;q,\alpha}(\vec{x};\vec{s})\big] = \prod_{k=1}^{n}\big(q^{(1-x_{k})\mathbf{1}_{x_{k}\leq 1 }}-q^{k-1}\big).
\end{equation}
If $x_1\geq 1$ then the $k=1$ term
%$$
%q^{-x_{k+1}} -\alpha^{-1}q^{2k} - q^{k}\big(q^{\frac{-s_{x_{k+1}}-x_{k+1}}{2}}-\alpha^{-1} q^{\frac{s_{x_{k+1}}-x_{k+1}}{2}}\big)
%$$
in the product \eqref{eq.Zs} defining $Z_{n;q,\alpha}$ evaluates to zero (recall that for $x\geq 1$,  $s_{x}=x$). Thus, when $x_1\geq 1$, we find that \eqref{eq.alphaindhalfrestate} holds (the left-hand side is zero as we just observed, and the right-hand side is clearly zero as well).
Owing to the ordering of $\vec{x}$, it now remains to show that \eqref{eq.alphaindhalfrestate} holds under the assumption that all $x_i\leq 0$. In this case, the right-hand side of \eqref{eq.alphaindhalfrestate} simplifies to
$$
\prod_{k=1}^{n} \big(q^{(1-x_{k})}-q^{k-1}\big).
$$

Making the change of variables $t_k = 1-x_k$ for $k=1,\ldots,n$, $S(t) = s_{1-t}$ for $t\geq 0$, and setting
$$
G(t,k):=q^{t-1} -q^{2k}\alpha^{-1} - q^k q^{(t-1)/2}\alpha^{-1/2} \f(S(t)) = q^{t-1} (1-q^{\frac{-S(t)+1-t}{2} +k}) (1+ q^{\frac{S(t)+1-t}{2}+k}\alpha^{-1}),
$$
(recall from \eqref{eq.f} that $\f(s) = q^{-s/2} \alpha^{1/2} - q^{s/2}\alpha^{-1/2}$)
the proof of \eqref{eq.alphaindhalfrestate} reduces to showing that for $0\leq t_1\leq \cdots \leq t_n$,
\begin{equation}\label{eq.lefttoshow}
\frac{\alpha^n}{q^{n(n-1)/2}} \EE\Big[\prod_{k=1}^{n} G(t_{k},k-1)\Big] = \prod_{k=1}^{n} (q^{t_{k}}-q^{k-1}).
\end{equation}

We proceed similarly to the $q\to 1$ case. The right-hand side of \eqref{eq.lefttoshow} is a multinomial in the variables  $\{q^{t_1},\ldots, q^{t_n}\}$ with maximal degree 1 in each $q^{t_i}$ term. The multinomial vanishes when $q^{t_{k+1}}=q^k$ and has coefficient $1$ in front of the maximal total degree term $q^{t_{1}}\cdots q^{t_n}$. This is a complete characterization of the multinomial, hence it suffices to show that the left-hand side of \eqref{eq.lefttoshow} likewise satisfies these properties. If (as in the $q\to 1$ case) we set  $y(t) = (S(t)+t-1)/2$ then $G(t,k)$ contains a factor $(1-q^{-y(t) +k})$, or equivalently (up to a multiplicative factor) $q^{y(t)}-q^k$. This is like the factor $y(t)-k$ in the $q\to 1$ case, and the same argument shows that if $q^{t_{k}}=q^{k-1}$ (or equivalently $t_{k}=k-1$) for some $k\in \{1,\ldots, n\}$ then the expectation on the left-hand side of \eqref{eq.lefttoshow} will be zero.

It remains to prove the multinomiality and highest degree coefficient for the left-hand side. The expectation on the left-hand side of \eqref{eq.lefttoshow} can be expanded as
\begin{equation}\label{eq.binth}
\EE\Big[\prod_{k=1}^{n} G(t_{k},k-1)\Big]  = \sum_{I\subseteq \{1,\ldots ,n\}} \EE\Big[\prod_{k\in I} \big(- q^{k-1} q^{(t_k-1)/2}\alpha^{-1/2} \f(S(t_k))\big)\Big] \, \prod_{k\in I^c} \big(q^{t_k-1} -q^{2(k-1)}\alpha^{-1} \big).
\end{equation}
The expectation in the right-hand side above boils down to computing $\EE\Big[\prod_{k\in I} \f\big(S(t_k)\big)\Big]$.

This is where Lemmas \ref{lem.eign} and \ref{sumid} come into play. In particular, observe that using the three term recurrence relation \eqref{eq.threeterm} for continuous $q^{-1}$-Hermite polynomials, along with Lemma \ref{lem.eign} we find that (recall also Definition \ref{def.halfstat}) for any $\ell\geq 0$ and $t\geq t'\geq 0$
\begin{equation}\label{eq.keyit}
\EE\Big[\f(S(t)) h_{\ell}\big(\f(S(t))/2 \barspace q\big) \big\vert S(t')\Big] = q^{(t-t')\frac{\ell+1}{2}} h_{\ell+1}\big(\f(S(t'))/2\barspace q\big) + (q^{-\ell}-1)q^{(t-t')\frac{\ell-1}{2}} h_{\ell-1}\big(\f(S(t'))/2\barspace q\big).
\end{equation}
Recalling that $h_0(x\barspace q) =1$, the above relation (along with sequential evaluation of conditional expectations) provides an immediate route to computing $\EE\Big[\prod_{k\in I} \f\big(S(t_k)\big)\Big]$. What is important to us is not the exact computation, but rather the readily confirmed (by induction on $|I|$) fact that $\EE\Big[\prod_{k\in I} q^{t_k/2}\f\big(S(t_k)\big)\Big]$ is a multinomial in $\{q^{t_k}\}_{k\in I}$ variables, with degree at most 1 in each variable. Indeed, this follows from the fact that \eqref{eq.keyit} only shifts the index $\ell$ up or down by exactly 1. This observation, combined with \eqref{eq.binth} proves the multinomiality of the left-hand side of \eqref{eq.lefttoshow} and that its maximal degree is 1 in each variable $\{q^{t_1},\ldots q^{t_n}\}$.

Again, using \eqref{eq.binth} along with \eqref{eq.keyit} one sees that the coefficient of the $q^{t_1}\cdots q^{t_{n}}$ term in the right-hand side of \eqref{eq.lefttoshow} is
$$
\frac{\alpha^n}{q^{n(n-1)/2}}\sum_{I\subseteq \{1,\ldots ,n\}} \prod_{k\in I} \big(-q^{k-1}(q\alpha)^{-1/2}\big) h_{|I|}\big(\f(1)/2\big) \prod_{k\in I^c} (q^{-1}).
$$
This term comes from always choosing the first term on the right-hand side of \eqref{eq.keyit} (otherwise certain $q^{t_i}$ variables will be missing).
Letting $j=|I|$, we can rewrite the above as
\begin{equation}\label{eq.almost}
\frac{\alpha^n}{q^{n(n-1)/2}}\sum_{j=0}^{n} (-1)^j (q\alpha)^{-j/2} h_{j}\big(\f(1)/2\big) q^{j-n} \sum_{\substack{I\subseteq \{1,\ldots ,n\}\\ |I|=j}} \prod_{k\in I} q^{k-1}.
\end{equation}
Recalling that
$$
\sum_{\substack{I\subseteq \{1,\ldots ,n\}\\ |I|=j}} \prod_{k\in I} q^{k-1} = q^{j(j-1)/2} {n \choose j}_q
$$
we find that \eqref{eq.almost} is exactly the same as the left-hand side of \eqref{eq.sumid}, which by Lemma \ref{sumid} is equal to 1. This proves that the highest degree coefficient is 1 and hence shows \eqref{eq.lefttoshow} and completes the proof of \eqref{eq.alphaindhalf}.
%This relation can be encoded in terms of a matrix operator. Let $\big[M_{\tau',\tau}(i,j)\big]_{i,j\in \Z_{\geq 0}$ denote the tridiagonal matrix indexed by $\Z_{\geq 0}$ whose row $i$ columun $j$ entry is given by
%$$
%M_{\tau',\tau}(i,j) =
%\begin{cases}
%q^{(\tau-\tau')\frac{i-1}{2}}&\textrm{if } i=j+1,\\
%(q^{-i}-1)q^{(\tau-\tau')\frac{i-1}{2}}&\textrm{if } i=j-1,\\
%&\textrm{otherwise}.
%\end{cases}
%$$
\end{proof}

\begin{proof}[Proof of Lemma \ref{lem.eign}]
Recall the operator $K$ from \eqref{eq.keq} which we reproduce here:
$$
(Kg)(s) = \frac{q^{s}}{\alpha + q^{s}} g(s+1) + \frac{\alpha}{\alpha + q^{s}} g(s-1).
$$
This operator acts on function $g:\Z\to \R$. However, one sees that it really takes functions defined on the even integers and maps them to functions on the odds, and {\it vice versa}. Define $\Z^e=2\Z$ and $\Z^o=2\Z+1$ to be the even and odd integers. Then let
$K^{e\leftarrow o}$ be the restriction of $K$ acting on function $g:\Z^o\to \R$ and returning functions $K^{e\leftarrow o}g:\Z^e\to \R$. Likewise, define $K^{o\leftarrow e}$ which acts from even integers functions to odd integers function. Let
$$
L^e = K^{e\leftarrow o}K^{o\leftarrow e},\qquad L^o = K^{o\leftarrow e}K^{e\leftarrow o},
$$
and observe that $L^e$ acts on function $g:\Z^e\to \R$ and returns the same type of function (and {\it vice versa} for $L^o$ but with $\Z^o$). $L^e$ and $L^o$ can be considered as restrictions of the operator $L$ which acts on functions $g:\Z\to \R$ as
\begin{eqnarray*}
(Lg)(s) &=& \frac{\alpha}{\alpha+q^s}\frac{\alpha}{\alpha+q^{s-1}} g(s-2) + \Big(1- \frac{\alpha}{\alpha+q^s}\frac{\alpha}{\alpha+q^{s-1}} - \frac{q^s}{\alpha+q^s} \frac{q^{s+1}}{\alpha+q^{s+1}} \Big) g(s)\\ &&+ \frac{q^s}{\alpha+q^s} \frac{q^{s+1}}{\alpha+q^{s+1}}g(s+2).
\end{eqnarray*}

Define $\psi_{n}(s):\Z\to \R$ by
$$
\psi_{n}(s) = h_n\big(\f(s)/2\barspace q\big)
$$
(where $\f(s) = q^{-s/2} \alpha^{1/2} - q^{s/2}\alpha^{-1/2}$ as in \eqref{eq.f}) and let $\psi^{e}_n$ and $\psi^{o}_n$ be the restrictions of $\psi_n$ to $s\in \Z^e$ and $\Z^o$ respectively. It follows from \cite[Equation (2.8)]{CK} that
\begin{equation}\label{eq.eignL}
(L\psi_n)(s) = q^n \psi_n(s), \qquad \textrm{or equivalently } \qquad (L^e\psi^e_n)(s) = q^n \psi^e_n(s) \quad \textrm{and} \quad (L^o\psi^o_n)(s) = q^n \psi^o_n(s).
\end{equation}
To get the above eigenrelation from \cite[Equation (2.8)]{CK}, first take $\beta=0$. Then take $e^{-y} = \tilde\alpha q^\ell$ as in \cite[Section 3]{CK}. (In fact, in \cite{CK} they use $\alpha$ where we have just used $\tilde\alpha$. We used $\alpha$ for a different purpose, so we replace their $\alpha$ with $\tilde\alpha$.) Letting $\ell=\tfrac{1}{2}s$ and $\tilde\alpha = \alpha^{-1/2}$ yields the relation for $L^e$, and letting $\ell=\tfrac{1}{2}s+\tfrac{1}{2}$ and $\tilde\alpha = q^{-1/2}\alpha^{-1/2}$ yields the relation for $L^o$.

Now define $\psi^{o\leftarrow e}_n(s) = \big(K^{o\leftarrow e}\psi^e_n\big)(s)$ for $s\in \Z^o$ and $\psi^{e\leftarrow o}_n(s) = \big(K^{e\leftarrow o}\psi^o_n\big)(s)$ for $s\in \Z^e$.
Lemma \ref{lem.eign} is equivalent to the claim that
\begin{equation}\label{eq.boils}
\psi^{o\leftarrow e}_n(s) = c^o_n \psi^o_n(s),\qquad \psi^{e\leftarrow o}_n(s) = c^e_n \psi^e_n(s),\qquad \textrm{where} \quad c^o_n=c^e_n = q^{n/2}.
\end{equation}
We will prove \eqref{eq.boils} in two steps. First we will establish that $\psi^{o\leftarrow e}_n(s) = c^o_n \psi^o_n(s)$ for some constant $c^o_n$ (and likewise for the other term), then we will compute the constant.

Observe that by the eigenrelation \eqref{eq.eignL},
$$
\big(K^{o\leftarrow e} \psi^{e\leftarrow o}_n\big)(s)  = \big(L^o \psi^o_n\big)(s) = q^n \psi^o_n(s)
$$
Thus, applying $K^{e\leftarrow o}$ to the above, we find that
\begin{equation}\label{eq.learrow}
\big(L^e \psi^{e\leftarrow o}_n\big)(s) = q^n \psi^{e\leftarrow o}_n(s).
\end{equation}

In what follows we will work only with the even case -- the odd case follows verbatim replacing $e$ by $o$. Define $m(s) = \alpha q^{-s(s-1)/4} (\alpha+q^s)^{-1/2}$ and an associated weighted inner product $(f,g)_{\Z^e;m} = \sum_{s\in \Z^e} f(s) g(s) \big(m(s)\big)^{-2}$. The set of all $f$ with $(f,f)_{\Z^e;m}<\infty$ defines a weighted $L^2$ space. We claim that the only functions $\psi:\Z^e\to \R$ in this weighted $L^2$ space for which $\big(L^e \psi\big)(s)=q^n \psi(s)$ are constants multiple of $\psi^e_n$. To prove this we follow the approach of \cite[Section 3]{CK} and conjugate $L$ to be self-adjoint. Define a diagonal multiplication operator $M$ which takes $g(s)$ to $m(s)g(s)$. Consequently, $L^{e,M}=M^{-1} L^e M$ is a symmetric operator with respect to the standard (unweighted) pairing $(f,g)_{\Z^e} = \sum_{s\in \Z^e} f(s)g(s)$. Letting $\psi^{e,M}_n(s) = M^{-1} \psi^e_n$, we have $\big(L^{e,M}\psi^{e,M}_n\big)(s) = q^n \psi^{e,M}_n(s)$. This, together with the self-adjointness and the fact that all eigenvalues are distinct implies that $\big(\psi^{e,M}_n,\psi^{e,M}_m\big)_{\Z^e} = \delta_{m,n} a_n$ for some constant $a_n$. In fact, this orthogonality statement is already given in \cite[Equation (4.1)]{CK} (see also Remark \ref{rem.eqnfo}). Moreover, \cite[Theorem 3.9]{CK} shows that the $\big\{\psi^{e,M}_n\big\}_{n\geq 0}$ form a complete orthogonal basis of the $L^2$ space defined with respect to the above inner product. That theorem is stated for the weighted inner product above and works directly with the $\psi^{e}_n$. Note also, that to translate that result one must use the same substitutions described earlier in this proof.

%We claim that the only function $\psi:\Z^e\to \R$ for which $\big(L^e \psi\big)(s)=q^n \psi(s)$ is a constant multiple of $\psi^e_n$. To prove this we follow the approach of \cite[Section 3]{CK} and conjugate $L$ to be self-adjoint. Define a diagonal multiplication operator $M$ which takes $g(s)$ to $m(s)g(s)$. Then letting $m(s) = \alpha q^{-s(s-1)/4} (\alpha+q^s)^{-1/2}$ results in a symmetric operator $L^{e,M}=M^{-1} L^e M$ with respect to the standard pairing $(f,g)_{\Z^e} = \sum_{s\in \Z^e} f(s)g(s)$. (We will work here only with the even case as the odd case follows verbatim with $e$ replaced by $o$.) Letting $\psi^{e,M}_n(s) = M^{-1} \psi^e_n$, we have $\big(L^{e,M}\psi^{e,M}_n\big)(s) = q^n \psi^{e,M}_n(s)$. This, together with the self-adjointness and the fact that all eigenvalues are distinct implies that $\big(\psi^{e,M}_n,\psi^{e,M}_m\big)_{\Z^e} = \delta_{m,n} a_n$ for some constant $a_n$. In fact, this orthogonality statement is already given in \cite[Equation (4.1)]{CK} (see also Remark \ref{rem.eqnfo}). Moreover, \cite[Theorem 3.9]{CK} shows that the $\big\{\psi^{e,M}_n\big\}_{n\geq 0}$ form a complete orthogonal basis of the $L^2$ space defined with respect to the above inner product. That theorem is stated for the weighted inner product one gets by including $M$ into the weight and working directly with the $\psi^{e}_n$. Note also, that to translate that result one must use the same substitutions described earlier in this proof.

%$\psi=\psi^{e\leftarrow o}_n$

Now consider $\psi$ such that $\big(L^e \psi\big)(s)=q^n \psi(s)$ and such that $\psi$ is in the weighted $L^2$ space defined in the proceeding paragraph. Then letting $\psi^{M} = M^{-1}\psi$ we find that
$$
q^n \big(\psi^M,\psi^{e,M}_m\big)_{\Z^e} = \big(L^{e,M} \psi^M,\psi^{e,M}_m\big)_{\Z^e} = \big(\psi^M,L^{e,M}\psi^{e,M}_m\big)_{\Z^e} = q^m \big(\psi^M,\psi^{e,M}_m\big)_{\Z^e}
$$
which shows that $\big(\psi^M,\psi^{e,M}_m\big)_{\Z^e}$ is zero unless $n=m$. By assumption, the function $\psi^{M}$ has finite $L^2$ norm (with respect to the standard inner product). Thus, by the completeness and orthogonality of the $\psi^{e,M}_n$, this implies the desired relation that $\psi^{M}$ must be a finite constant times $\psi^{e,M}_n$ (and the same statement holds with the $M$ terms removed).
Hence, from \eqref{eq.learrow} and the above deduction it follows that $\psi^{e\leftarrow o}_n(s) = c^e_n \psi^e_n(s)$ for some finite constant $c^e_n$ (and likewise with $e$ and $o$ switched).

It remains to determine the constant $c^e_n$ and $c^o_n$. These are determined directly by comparing coefficients. Note that $h_n(x\barspace q)$ has top degree $x^n$ coefficient $2^n$ (as follows readily from the recursion \eqref{eq.threeterm} defining it). Thus, $\psi^e_n(s)$, expressed as a polynomial in $q^{s/2}$, has top degree $(q^{s/2})^n$ coefficient equal to $(-\alpha^{-1})^n$. Using this, we can compute the top degree coefficient in $q^{s/2}$ of $(\alpha+ q^s) \psi^{e\leftarrow o }_n(s)$ in two ways. The first uses $\psi^{e\leftarrow o}_n(s) = c^e_n \psi^e_n(s)$ to deduce that the coefficient of $(q^{s/2})^{n+2}$ (which is the top degree) is $c^e_n (-\alpha^{-1})^n$. The other way uses the definition
$$
(\alpha+ q^s) \psi^{e\leftarrow o}_n(s) = \alpha \psi^o_n(s) + q^s \psi^o_n(s+1)
$$
to show that the coefficient of $(q^{s/2})^{n+2}$ is $q^{n/2} (-\alpha^{-1})^n$ (this comes from the $q^s$ term on the right-hand side above). Matching coefficients gives $c^e_n = q^{n/2}$. Similarly we determine $c^o_n=q^{n/2}$. This shows \eqref{eq.boils} and hence proves the lemma.

\end{proof}

\begin{proof}[Proof of Lemma \ref{sumid}]
We prove Lemma \ref{sumid} for $q>1$ below. The result for $q<1$ follows immediately from analytic continuation in $q$.

We rely on a known identity for the continuous $q$-Hermite polynomials $\{H_n(x\barspace q)\}_{n\geq 0}$ which are defined through an analogous recursion relation to $h_n(x\barspace q)$ as
$$
2x H_n(x\barspace q) = H_{n+1}(x\barspace q) + (1-q^{n}) H_{n-1}(x\barspace q), \quad \textrm{with} \quad H_{-1}(x\barspace q)=0,\quad \textrm{and}\quad H_{0}(x\barspace q)=1.
$$
Comparing this to the relation \eqref{eq.threeterm} for $h_n$ one readily confirms that (we use $\I=\sqrt{-1}$)
$$
h_n(x\barspace q) = \I^{-n} H_n(\I x\barspace q^{-1}).
$$

We rely on identity (3.26.12) from \cite{KoekoekSwarttouw} which states that for $x= \cos(\theta)$ and $\gamma$ arbitrary
\begin{equation}\label{eq.kseq}
G(\gamma,q,x,t):=\sum_{j=0}^{\infty} \frac{(\gamma;q)_j}{(q;q)_j} H_j(x\barspace q)t^j  = \frac{(\gamma e^{\I \theta}t;q)_{\infty}}{(e^{\I \theta}t;q)_{\infty}}\,\pFq{2}{1}{\gamma,0}{\gamma e^{\I \theta} t}{q}{e^{-\I \theta}t},
\end{equation}
where ${}_2 \phi_1$ is the basic hypergeometric function. Rewriting this expression in terms of $h_j$ instead of $H_j$ and substituting $q^{-1}$ for $q$ we find that
$$
G(\gamma,q^{-1},x,t)=\sum_{j=0}^{\infty} \frac{(\gamma;q^{-1})_j}{(q^{-1};q^{-1})_j} \I^j h_j(x/\I\barspace q)t^j.
$$
The above formula is valid so long as $q>1$.

By setting $\gamma =q^n$ we cut off the infinite sum on the right-hand side so as to be zero after $j=n$ (since $(q^n;q^{-1})_j$ for $j>n$).
% so that
%$$
%G(q^n,q^{-1},x,t)=\sum_{j=0}^{n} \frac{(q^n;q^{-1})_j}{(q^{-1};q^{-1})_j} \I^j h_j(x/\I|q)t^j.
%$$
Further, using the relations
$$
(q^n;q^{-1})_j = \frac{(q;q)_n}{(q;q)_{n-j}} \qquad \textrm{and}\qquad \frac{1}{(q^{-1};q^{-1})_j} = \frac{(-1)^j q^{j(j+1)/2}}{(q;q)_j},
$$
we can rewrite
$$
G(q^n,q^{-1},x,t)=\sum_{j=0}^{n} (-1)^j {n \choose j}_q q^{j(j+1)/2} \I^j h_j(x/\I\barspace q)t^j.
$$
Finally, taking $x=\I \f(1)/2$ and $t= -\I (q\alpha)^{-1/2}$ we find that
$$
G(q^n,q^{-1},\I \f(1)/2,-\I (q\alpha)^{-1/2}) =\sum_{j=0}^{n} (-1)^j {n \choose j}_q q^{j(j+1)/2} (q\alpha)^{-j/2} h_j\big(\f(1)/2\barspace q\big)
$$
where the right-hand side above now coincides with the summation in \eqref{eq.sumid}.

In light of this relation, it remains to verify that
\begin{equation}\label{eq.tover}
G\big(q^n,q^{-1},\I \f(1)/2,-\I (q\alpha)^{-1/2}\big) = \frac{q^{n(n+1)/2}}{\alpha^n}.
\end{equation}
The relation $\I \f(1)/2 = \cos(\theta)=\tfrac{1}{2}(e^{\I \theta}+e^{-\I \theta})$ yields two choices for $e^{\I \theta}$, namely $e^{\I \theta} = \I (\alpha/q)^{1/2}$ and  $e^{\I \theta} = -\I (q/\alpha)^{1/2}$. We take the second of these choices, which impies that  $e^{\I \theta}t=-\I\alpha^{-1}$ and $e^{-\I \theta}t=q^{-1}$. Thus, from \eqref{eq.kseq}, we find that that
$$
G\big(q^n,q^{-1},\I \f(1)/2,-\I (q\alpha)^{-1/2}\big)= \frac{(-q^n \alpha^{-1};q^{-1})_{\infty}}{(-\alpha^{-1};q^{-1})_{\infty}}\,\pFq{2}{1}{q^n,0}{-q^n \alpha^{-1}}{q^{-1}}{q^{-1}}.
$$
It follows from identity (1.5.3) of \cite{GasperRahman} that (after changing $q$ to $q^{-1}$ and setting $b=0$)
% we have that (changing $q$ to $q^{-1}$)
%$$
%\pFq{2}{1}{q^{n},b}{c}{q}{q} = \frac{(c/b;q^{-1})_n}{(c;q^{-1})_n} b^n.
%$$
%Taking $b\to 0$ yields
$$
\pFq{2}{1}{q^{n},0}{c}{q}{q} = \frac{(-c)^n q^{(n-1)(n-2)/2}}{(c;q^{-1})_n}.
$$
Using this we conclude that
$$
G\big(q^n,q^{-1},\I \f(1)/2,-\I (q\alpha)^{-1/2}\big)= \frac{(-q^n \alpha^{-1};q^{-1})_{\infty}}{(-\alpha^{-1};q^{-1})_{\infty}} \frac{\alpha^{-n} q^{n(n+1)/2}}{(-q^n\alpha^{-1};q^{-1})_{n}} =  \frac{q^{n(n+1)/2}}{\alpha^n},
$$
where the last equality uses $(-q^n \alpha^{-1};q^{-1})_{\infty}= (-q^n\alpha^{-1};q^{-1})_{n}(-\alpha^{-1};q^{-1})_{\infty}$. This proves \eqref{eq.tover} and hence completes the proof of the lemma.
\end{proof}

%\bibliographystyle{plain}
%\bibliography{Biblio}

\end{document}